\documentclass[12pt]{article}
\usepackage{amsfonts}
\usepackage[all]{xy}

\newtheorem{lemma}{Lemma}
\newtheorem{theorem}{Theorem}
\newtheorem{proposition}{Proposition}
\newtheorem{definition}{Definition}
\newenvironment{proof}{{\bf Proof}.}{\hfill $\Box$}
\newtheorem{remark}{Remark}
\newtheorem{corollary}{Corollary}

\begin{document}
\title{ $C^*$--non--linear second quantization}
\author{Luigi Accardi,  Ameur Dhahri\\\vspace{-2mm}\scriptsize Volterra Center, University of Roma Tor Vergata\\\vspace{-2mm}\scriptsize Via Columbia 2, 00133 Roma, Italy\\\vspace{-2mm}\scriptsize e-mail:accardi@volterra.uniroma2.it\\\scriptsize ameur@volterra.uniroma2.it}

\date{}
\maketitle

\begin{abstract}
We construct an inductive system of $C^*$--algebras each 
of which is isomorphic to a finite tensor product of copies of the one 
mode $n$--th degree polynomial extension of the usual Weyl algebra 
constructed in our previous paper \cite{[AcDha]}. \\
We prove that the inductive limit $C^*$-algebra is factorizable and has a 
natural localization given by a family of $C^*$-sub-algebras each of 
which is localized on a bounded Borel subset of $\mathbb{R}$. 
Finally we prove that the corresponding families of Fock 
states, defined on the inductive family of $C^*$-algebras, is projective 
if and only if $n=1$. 
 This is a weak form of the no--go theorems which
emerge in the study of representations of current algebras over Lie algebras.
\end{abstract}

\tableofcontents

\section{Introduction: the $C^*$--non--linear quantization program}



The present paper is a contribution to the program
of constructing a theory of renormalized higher powers
of quantum white noise (RPWN) or equivalently of
non relativistic free Boson fields.\\
This program has an old history, but the approach discussed here 
started in 1999 with the construction of the Fock representation 
for the renormalized square of Boson white noise \cite{[AcLuVo99]}.
This result motivated a large number of papers extending it
in different directions and exhibiting connections with almost all
fields of mathematics, see for example \cite{[Snia00]} for the case of
free white noise, \cite{[AcFrSk00]} for the connection with infinite 
divisibility and for the identification of the vacuum distributions 
of the generalized fields with
the three non standard Meixner classes, \cite{[AcAmFr02]} and
\cite{[Prohor05]} for finite temperature representations,
\cite{[AcDha10]} for the construction of the Fock functor,
the survey \cite{[AcBouk09sg]} and the paper
\cite{[AcBou09Cohom]} for the connections with conformal
field theory and with the Virasoro--Zamolodchikov hierarchy,
\cite{[AcBou12_padic]} for the connections between
renormalization and central extensions.\\
The problem is the following. One starts with the Schroedinger representation
of the Heisenberg real $*$--Lie algebra with skew--adjoint generators $iq$
(imaginary unit times position), $-ip$ ($-i$ times momentum), $E:=i1$
($i$ times central element) and relations $[iq,-ip]=i1$.\\
The universal enveloping algebra of this Lie algebra, called
for brevity the {\it full oscillator algebra} (FOA),
can be identified with the algebra of differential operators
in one real variable with complex polynomial coefficients.\\
The continuous analogue of the Heisenberg Lie algebra is the
non relativistic free boson field algebra, also called the current
algebra over ${\mathbb R}$ of the Heisenberg algebra,
whose only non zero commutation relations are, in the sense of operator
valued distributions:
$$
[q_{s},p_{t}]=\delta(s-t)1 \qquad;\qquad s,t\in\mathbb{R}
$$
The notion of current algebra has
been generalized from the Heisenberg algebra to more general
$*$--Lie algebras (see Araki's paper \cite{[Arak69]} for
a mathematical treatment and additional references): in this case
the self--adjoint generators of the Cartan sub--algebras are
called generalized fields.\\
Notice that the definition of current algebra of a given Lie algebra
is independent of any representation of this algebra, i.e. it does not
require to fix a priori a class of states on this algebra.\\
One can speak of \textit{$*$--Lie algebra second
quantization} to denote the transition from the construction of
unitary representations of a $*$--Lie algebra to the
construction of unitary representations of its current
algebra over a measurable space (typically
$\mathbb{R}$ with its Borel structure).\\
Contrarily to the discrete case,
the universal enveloping algebra of the current algebra
over ${\mathbb R}$ of the Heisenberg algebra is ill
defined because of the emergence of higher powers of
the $\delta$--function.
This is the mathematical counterpart of the old problem of
defining powers of local quantum fields.\\
\textit{Any rule that, giving a meaning to these powers,
defines a $*$--Lie algebra structure, is called a
renormalization procedure}. The survey \cite{[AcBouk09sg]}
describes two inequivalent renormalization procedures
and the more recent paper  \cite{[AcBou12_padic]}
shows the connection between them.\\
\textit{The second step of the program, after renormalization,
is the construction of unitary representations of the
resulting $*$--Lie algebra}. This step, which is the most
difficult one because of the no--go theorems (see discussion
below), is usually done by fixing a state and considering
the associated cyclic representation.
At the moment, even in the first order case, i.e. for
usual fields, the explicitly constructed representations
are not many, they are essentially reduced to gaussian
(quasi--free) representations.
Moreover any gaussian representation can be obtained,
by means of a standard construction, from the Fock
representation which is characterized by the property that
the cyclic vector, called vacuum, is in the kernel of
the annihilation operators.\\ This property has been taken as an
heuristic principle to define the notion of \textit{Fock state}
also in the higher order situations
(see \cite{[AcBouk09sg]} for a precise definition).\\
It can be proved that, for all renormalization procedures
considered up to now, the Fock representation and the Fock
state are factorizable, in the sense of Araki--Woods
\cite{[ArWoo66]}.
This property poses an obstruction to the existence of such
representation, namely that the restriction of the Fock state
on any factorizable Cartan sub--algebra must give
rise to a classical infinitely divisible process.
If this is not the case then \textit{no Fock representation,
and more generally no cyclic representation associated to a
factorizable state, can exist.\\
When this is the case we say that a no--go theorem holds.}
Nowadays several instances of no--go theorems are available.
The simplest, and probably most illuminating one, concerns
the Schroedinger algebra, which is the Lie algebra
generated by the powers $\leq 2$ of $p$ and $q$
(see \cite{[Snia00]}, also \cite{[AcFrSk00]}
and \cite{[AcBou05aHPqDef]} for stronger results).
This result implies that there is \textit{no natural analogue
of the Fock representation for the current algebra over
${\mathbb R}$ (for any $d\in{\mathbb N}$) of the FOA}.\\
On the other side we know (see \cite{[AcFrSk00]}, \cite{[AcLuVo99]}
and the above discussion) that \textit{for some sub--algebras
of current algebras of the FOA such a representation exist}.
This naturally rises the problem to characterize these
sub--algebras. \\
Since a full characterization at the moment is not available,
a natural intermediate step towards such a characterization is
to produce nontrivial examples.

To this goal a family of natural candidates is provided by
the $*$--Lie--sub--algebras of the FOA consisting of the real linear
combinations of the derivation operator and the polynomials
of degree less or equal than a fixed natural integer $n$.
Thus the generic element of such an algebra has the form
$$
up+P(q)  \qquad ;\qquad u\in\mathbb R
$$
where $P$ is a polynomial of degree $n$ with real coefficients.
For $n=1$ one finds the Heisenberg algebra; for $n=2$ the
Galilei algebra and, for $n>2$, some nilpotent Lie
algebras well studied in mathematics \cite{Bourbaki}, \cite{Fein2}
\cite{Franz}, \cite{Ovando}
but up to now, with the notable exception of the Galilei algebra
($n=2$), not considered in physics. \\
These $*$--Lie--algebras enjoy two very special properties:
\begin{enumerate}
\item[(i)] no renormalization is required in the definition of
the associated current algebra over ${\mathbb R}$;
\item[(ii)] in the Schroedinger representation of the FOA
the skew--adjoint elements of these sub--algebras can
be explicitly exponentiated giving rise to a nonlinear
generalization of the Weyl relations and of the corresponding
Heisenberg group. This was done in the paper \cite{[AcDha]}.
\end{enumerate}
Property (i) supports the hope of the existence of the Fock
representation for the above mentioned current algebra.
A direct proof of this fact could be obtained by proving the
infinite divisibility of all the vacuum characteristic
functions of the generalized fields.
Unfortunately even in the case $n=2$, in which this function
can be explicitly calculated, a direct proof of infinite
divisibility can be obtained only for a subset of the parameters
which define the generalized fields, but not for all, and this
problem is challenging the experts of infinite divisibility
since several years.\\
In the present paper we exploit property (ii) and the following
heuristic considerations are aimed at making a bridge between the
mathematical construction below and its potential physical interpretation.\\
Our goal is to consruct a $C^*$--algebra whose generators can be 
{\it naturally} identified with the following 
formal expressions that we call the {\it non--linear Weyl operators}:
\begin{equation}\label{form-expr-NLW}
e^{i\big(p(f_0)+q^n(f_n)+\dots+q^1(f_1)\big)}
=e^{i\sum_{j\in \{0,1,\dots,n+1\}}L_j(f_j)}
\end{equation}
The formal generators of  the non--linear Weyl operators 
(called {\it non--linear fields}) are heuristically expressed as 
{\it powers of the standard quantum white noise 
(or free Boson field)}, i.e. the pair of operator valued 
distributions $q_t,p_t$ with commutation relations
$$
[q_s,p_t] = i\delta (t-s)
$$
in the following way:
$$
L_{n+1}(f_0):=p(f_0)=\int_{\mathbb{R}}f_0(t)p_tdt 
\qquad ;\qquad L_{0}(f):=\int_{\mathbb{R}}f_k(t)q_t^0dt
:= 1\cdot \int f(s) ds 
$$
\begin{equation}\label{form-expr-NLF}
L_{k}(f):=q^k(f_k)=\int_{\mathbb{R}}f_k(t)q_t^kdt
 \qquad ;\qquad k\in\{0,1,\dots, n\}
\end{equation}
When $n= 1$, the expressions (\ref{form-expr-NLF}) are well 
understood and various forms of second quantization are known.
For example one can prove the unitarity of the Fock representation
the exponentiability, inside it, of the generators 
(\ref{form-expr-NLF}) and the commutation relations satisfied by 
them. A different example, for $n= 1$, is provided by Weyl second 
quantization: in it, by heuristic calculations, one guesses 
the commutation relations that should be satisfied by 
any representation  of the exponentials (\ref{form-expr-NLW}) 
and then one proves the existence of a $C^*$--algebra which 
realizes these commutation relations.\\
In the present paper we apply this approach to give a meaning to 
the exponentials (\ref{form-expr-NLW}) and in this sense we speak of
{\it $C^*$--second quantization}.\\
To this goal we exploit the fact that, if $\pi$ is a finite 
Borel partition of a bounded Borel subset of $\mathbb{R}$, 
then there is a natural way to give a meaning to the generalized 
Weyl algebra with test functions constant on the sets of $\pi$.
This is based on the identification of this algebra with the tensor 
product of $|\pi|$ (cardinality of $\pi$) {\it rescaled copies} of 
the one mode generalized Weyl algebra (see section \ref{The-ind-lim}).
This identification strongly depends on the specific structure
of the Lie algebra considered (see section \ref{subsection3.1} below).
\\
Using this we construct an inductive system of $C^*$--algebras each
of which is isomorphic to a finite tensor product of copies of the one
mode generalized Weyl algebra \textbf{but the embeddings defining
the inductive system are not the usual tensor product embeddings}. \\
The $C^*$--algebra, obtained as inductive limit from the above
construction, is naturally intepreted as a
\textit{$C^*$--quantization}, over $\mathbb{R}$, of the
initial $*$--Lie algebra. 
This $C^*$--algebra has a \textit{localization}
given by a family of $C^*$--sub--algebras, each of which has a natural
localization on bounded Borel subset of $\mathbb{R}$.\\ 
Moreover this system of local algebras is factorizable in the 
sense of\\ Definition \ref{cont-tens-prod-prop} below.\\
With this construction the problem of constructing unitary
representations of the current algebra over $\mathbb{R}$ of the
initial $*$--Lie algebra is reduced to the problem of finding
representations of this $C^*$--algebra: the advantages of
this transition from unbounded to bounded case are well known
in the case of standard, first order, quantization.\\
In the last section of the paper it is shown that, although the
Fock state is defined on each of the $C^*$--algebras of the 
inductive family, the corresponding family of states is projective
if and only if $n=1$ (i.e. for the usual Weyl algebra). This
result can be considered as a $C^*$--version of the no--go
theorems proved in \cite{[Snia00]}, \cite{[AcFrSk00]},
\cite{[AcBou05aHPqDef]} for different algebras.

The basic construction of the present paper can be
extended to more general classes of $*$--Lie algebras 
(for example the $C^*$--algebras associated to the 
renormalized square of white noise (RSWN)) and
more general spaces (i.e.  $\mathbb{R}^d$ instead of $\mathbb{R}$).

\section{The $1$--mode $n$--th degree Heisenberg
$*$--Lie algebra $heis_\mathbb{R}(1,n)$ }

\begin{definition}\label{df-heis(1,n)}{\rm For $n\in\mathbb N^*$ the \textit{$1$--mode $n$--th degree Heisenberg algebra},\\
denoted $heis_\mathbb{R}(1,n)$, is the pair
$$
\{V_{n+2},(L_j)^{n+1}_{j=0}\}
$$
where:
\begin{enumerate}
\item[-] $V_{n+2}$ is a $(n+2)$--dimensional real $*$--Lie algebra;
\item[-] $(L_j)^{n+1}_{j=0}$ is a skew--adjoint linear basis of $V_{n+2}$;
\item[-] the Lie brackets among the generators are given by
\begin{eqnarray*}\label{CR-Heis1n1}
[L_i,L_j]=0\qquad ;\qquad \forall i,j\in\{0,1,\cdots,n\}
\end{eqnarray*}
\begin{eqnarray*}\label{CR-Heis1n2}
[L_{n+1},L_k]=kL_{k-1}
\qquad ;\qquad \forall k\in\{1,\cdots,n\}
 \ , \ L_{-1}:=0
\end{eqnarray*}
\end{enumerate}
}
\end{definition}
\begin{remark}{\rm
\begin{enumerate}
\item[1)] Multiplying each of the generators $(L_j)^{n+1}_{j=0}$ by 
a strictly positive number, one obtains a new basis $(L_j')^{n+1}_{j=0}$ 
of $V_{n+2}$ satisfying the new commutation relations
\begin{eqnarray*}\label{CR-Heis1n3}
[L'_{n+1},L'_k]=kd_kL'_{k-1}
\qquad ;\qquad \forall k\in\{0,\cdots,n\}
 \ , \ L'_{-1}:=0
\end{eqnarray*}
In this case we speak of \textit{a re--scaled copy of} 
the $1$--mode $n$--th degree Heisenberg algebra.
\item[2)] Denoting $\mathbb{R}_n[X]$ the the vector space of polynomials
in one indeterminate with real coefficients and degree less
or equal than $n$, the assignment of the basis $(L_j)^{n+1}_{j=0}$
uniquely defines the parametrization
$$
(u,(a_k)_{k\in\{0,1,\dots,n\}})
\in \mathbb{R}\times\mathbb{R}_n[X] \equiv \mathbb R^{n+2} \ \mapsto
 \ \ell_0 (u,P):= \\
$$
\begin{equation}\label{gen-el-1md-alg}
:= uL_{n+1} + \sum_{k\in\{0,1,\dots,n\}}a_kL_{k}
=:uL_{n+1}+P(L) \in heis_\mathbb{R}(1,n)
\end{equation}
of $heis_\mathbb{R}(1,n)$ by elements of $\mathbb{R}_n[X]$.
When no confusion is possible we will use the identification
\begin{equation}\label{id-el-coord}
\ell_0(u,P) \equiv 
(u,(a_k)_{k\in\{0,1,\dots,n\}})
\in \mathbb R^{n+2}
\end{equation}
\end{enumerate}
}\end{remark}

\section{The Schroedinger representation and the\\
polynomial Heisenberg group $\hbox{Heis}(1,n)$}\label{Schred-repr}

Let $p,q,1$ be the usual momentum, position and identity
operators acting on the one mode boson Fock space
\begin{equation}\label{Fock}
\mathcal{H}_1=\Gamma(\mathbb{C})=L^2(\mathbb{R})
\end{equation}
The maximal algebraic domain $\mathcal D_{max}$ (see \cite{[AcAyOu03]}),
consisting of the linear combinations of vectors of the form
$$
q^np^k\psi_z
\qquad; \qquad k,n\in\mathbb{N} \ , \ z\in\mathbb{C}
$$
where $\psi_z$ is the exponential vector associated to
$z\in\mathbb{C}$, is a dense subspace of $\Gamma(\mathbb{C})$
invariant under the action of $p$ and $q$ hence of all the
polynomials in the two non commuting variables $p$ and $q$.
In particular, for each $n\in\mathbb N$, the real linear span of
the set $\{i1,ip,iq,\dots ,iq^n\}$, denoted
$heis_\mathbb{R}(F,1,n)$, leaves invariant
the maximal algebraic domain $\mathcal D_{max}$. Hence the commutators
of elements of this space are well defined on this domain
and one easily verifies that they define a structure of $*$--Lie algebra
on $heis_\mathbb{R}(F,1,n)$.

\begin{lemma}{\rm   
In the above notations the map  
\begin{equation}\label{sk-ad-Heis1n-id-gen}
L_{n+1}\mapsto ip \quad , \quad L_0\mapsto i1 \quad , \quad
L_k\mapsto iq^k \quad ; \quad k\in \{1, \dots  , n\}
\end{equation}
admits a unique linear extension from $heis_\mathbb{R}(1,n)$ onto 
$heis_\mathbb{R}(F,1,n)$
which is a $*$--Lie algebra isomorphism
called \textit{the Schroedinger representation} of the
$n$--th degree Heisenberg algebra $heis_\mathbb{R}(1,n)$.
}\end{lemma}
\begin{proof} 
The linear space isomorphism property follows from the linear
independence of the set $\{1,p,q,\dots,q^n\}$.
The $*$--Lie algebra isomorphism property follows from direct computation.
\end{proof}

In \cite{[AcDha]} (Theorem 1) it is proved that the unitary operators
\begin{equation}\label{df-Fck-NLWO}
W(u,P):=e^{i(up+P(q))} \in\hbox{Un}(L^2(\mathbb{R}))
\quad;\quad  (u,P)\in\mathbb{R}\times\mathbb{R}_n[X]
\end{equation}
satisfy the following polynomial extension of the Weyl relations:
\begin{eqnarray}\label{n-Weyl-rels}
W(u,P)W(v,Q)= W((u,P)\circ (v,Q))
\;;\qquad \forall
(u,P),(v,Q)\in \mathbb R \times \mathbb{R}_n[X]
\end{eqnarray}
where
\begin{eqnarray}\label{1Npol-comp-law1}
(u,P)\circ(v,Q):= (u+v,T^{-1}_{u+v}(T_uP+T_vS_uQ))
\end{eqnarray}
and for any $u,w\in\mathbb{R}$, the linear operators
$T_w, S_u:\mathbb{R}_n[X]\rightarrow\mathbb{R}_n[X]$ are
defined by the following prescriptions:
\begin{eqnarray*}
T_w1&=&1\nonumber\\
T_w(X^k)&=&\sum_{h=0}^{k-1}\frac{k!}{(k+1-h)!h!}w^{k-h}X^h+X^k
\,;\qquad\forall k\in\{1,\dots,n\}\label{aa}\\
(S_uP)(X)&:=&P(X+u) \quad\hbox{translation operator on $\mathbb{R}_n[X]$}\nonumber
\end{eqnarray*}
Denote
\begin{eqnarray*}
\mathcal{W}_{F,1,n} := \hbox{norm closure in} \ \mathcal B(\Gamma(\mathbb{C}))
 \ \hbox{of the linear span of the operators (\ref{df-Fck-NLWO})}.
\end{eqnarray*}
Identity (\ref{n-Weyl-rels}) implies that $\mathcal{W}_{F,1,n}$ is a 
$C^*$--algebra.

In \cite{[AcDha]} it is proved that the composition law
(\ref{1Npol-comp-law1}) is a Lie group law on
$\mathbb{R}\times \mathbb{R}_n[X]$ whose Lie algebra is $heis_\mathbb{R}(1,n)$.
Since the elements of $heis_\mathbb{R}(1,n)$ are parametrized by the pairs
$(u,P)\in\mathbb{R}\times\mathbb{R}_n[X]$ it is natural to
introduce the following notation.
\begin{definition}{\rm (see \cite{[AcDha]})   
The \textit{$1$--mode $n$--th degree Heisenberg group} is the set
\begin{eqnarray}\label{el-Heis-gr}
\hbox{Heis}(1,n):=
\left\{e^{\ell_0(u,P)} \ : \ (u,P)\in\mathbb{R}\times\mathbb{R}_n[X]\right\}
\end{eqnarray}
with composition law
\begin{eqnarray*}\label{comp-law-Heis1n}
e^{\ell_0(u,P)}\circ e^{\ell_0(v,Q)}
:= e^{\ell_0((u+v,T^{-1}_{u+v}(T_uP+T_vS_uQ)))}
\end{eqnarray*}
}\end{definition}
The name $\hbox{Heis}(1,n)$ is motivated by the fact that,
for $n=1$,  $\hbox{Heis}(1,n)$\\ reduces to the usual
the $1$--mode Heisenberg group.

\section{The free group--$C^*$--algebra of $\hbox{Heis}(1,n)$}
\label{1-mde-ndeg-Weyl}

\begin{definition}\label{df-discr-grp-alg}{\rm
Let $G$ be a group. The free complex vector space generated by the set
$$
\left\{ W_g \ : \ g\in G \right\}
$$
has a unique structure of unital $*$--algebra defined by the prescription that
the map $g\mapsto W_g$ defines a unitary representation of $G$, equivalently:
\begin{eqnarray}
W_gW_h &:=& W_{gh}\nonumber
\qquad ; \qquad g,h\in G\\
(W_g)^*  &:= &W_{g^{-1}} \qquad ; \qquad g\in G\label{Wg*=Wg-1}\\
1 &:=& W_{e}\nonumber
\end{eqnarray}
The completion  of $\mathcal{W}^{0}(G)$ under the (minimal) $C^*$--norm
\begin{eqnarray*}\label{df-discr-reg-normW01n}
\|x\| := \sup \{ \|\pi(x)\| \ : \ \pi\in
\{\hbox{$*$--representations of } \
 G\} \  \} \quad;\quad  x\in \mathcal{W}^{0}(G)
\end{eqnarray*}
will be called the \textit{free group--$C^*$--algebra of $G$} and denoted
$\mathcal{W}(G)$.
}\end{definition}
\begin{remark}{\rm
Because of (\ref{Wg*=Wg-1}) a $*$--representation of $\mathcal{W}(G)$
maps the generators $W_g$ $(g\in G)$, into unitary operators.  
}\end{remark} 
\begin{remark}\label{Rem3}{\rm
If $G,G'$ are groups, then any group homorphism (resp. isomorphism) 
$ \alpha \ : \ G\to G'$, extends uniquely 
to a $C^*$--algebra  homorphism (resp. isomorphism)  
$\tilde \alpha \ : \ \mathcal{W}(G)\to \mathcal{W}(G')$ characterized 
by the condition
$$
\tilde \alpha ( W_g ):= W_{\alpha g}
\qquad ; \qquad g\in G 
$$
}\end{remark} 
\begin{definition}\label{df-1mde-nth-deg-Weyl-alg}{\rm
If $G=\hbox{Heis}(1,n)$, its free group--$C^*$--algebra is called
the \\ \textit{$1$--mode $n$--th degree Weyl algebra} and denoted
\begin{equation}\label{df-W0(1,n)}
\mathcal{W}^{0}_{1,n} := \mathcal{W}^{0}(\hbox{Heis}(1,n))
\end{equation}  
For its generators, called the \textit{$1$--mode $n$--th degree Weyl
operators}, we will use the notation
\begin{equation}\label{gen-el-Heis}
W^{0}(u,P) := W_{e^{\ell_0(u,P)}} \qquad;\qquad
(u,P)\in\mathbb{R}\times\mathbb{R}_n[X]
\end{equation}
}\end{definition}
By construction the map
\begin{equation}\label{bb}
u_F:W^{0}(u,P) \in \mathcal{W}^{0}_{1,n} \mapsto
W(u,P)\in\mathcal{W}_{F,1,n}
\end{equation}
where the operators
$W(u,P)$ are those defined in (\ref{df-Fck-NLWO}), is a group isomorphism.
Hence the definition of free group--$C^*$--algebra implies that it
can be extended to a surjective $*$--representation called
\textit{the Fock representation of $\mathcal{W}^{0}_{1,n}$}.\\
We will use the same symbol $u_F$ for this extension.

We conjecture that, in analogy with the case $n=1$,
the $*$--homomorphism of $\mathcal{W}^{0}_{1,n}$ onto $\mathcal{W}_{F,1,n}$
is in fact an isomorphism and that
there is a unique $C^*$--norm on $\mathcal{W}^{0}_{1,n}$.

\section{The current algebra of $heis_\mathbb{R}(1,n)$ over $\mathbb{R}$}
\label{section3}

Denote
$$
\mathcal{H}_0(\mathbb{R}):=L^1_{\mathbb{R}}(\mathbb{R})
\cap L^\infty_{\mathbb{R}}(\mathbb{R})
=\bigcap_{1\leq p\leq \infty }L^p_{\mathbb{R}}(\mathbb{R})
$$
$\mathcal{H}_0(\mathbb{R})$ has a natural structure of
real pre--Hilbert algebra with the pointwise operations
and the $L^2$--scalar product.
\begin{lemma}\label{Jacobi}{\rm
For any $*$--sub--algebra $\mathcal{T}$
of $\mathcal{H}_0(\mathbb{R})$ and $n\in\mathbb{N}$,
there exists a unique real $*$--Lie algebra with
skew--adjoint generators
\begin{eqnarray*}\label{df-lin-gen}
\left\{L_0 \ , \ L_k(f)\ : \
k\in\{1,\dots,n+1\} \ ; \ f\in\mathcal{T} \right\}
\end{eqnarray*}
where, with the notation
\begin{eqnarray}\label{df-q0c(f)}
L_0(f) := L_0\int_{\mathbb{R}}f(t)dt
\quad;\quad L_{-1}(f)=0 \quad;\quad  \forall f\in\mathcal{T}
\end{eqnarray}
the maps $f\mapsto L_k(f)$
($k\in\{0,1,\dots,n\}$)
are real linear on $\mathcal{T}$ and the Lie brackets are given,
for all $f,g\in\mathcal{T}$, by
\begin{eqnarray}\label{CCR1n-a}
[L_{i}(f),L_{j}(g)]=0
\qquad ;\qquad i,j\in\{0,1,\dots,n\}
\end{eqnarray}
\begin{eqnarray}\label{CCR1n-c}
[L_{n+1}(f),L_{k}(g)]=kL_{k-1}(fg)
\ ; \ k\in\{0,1,2,\dots,n\} \ , \ L_{-1}(f)=0
\end{eqnarray}
}\end{lemma}
\begin{proof}
By definition the Lie brakets of two generators
defined by (\ref{CCR1n-a}), (\ref{CCR1n-c})
is a multiple of the generators.
In order to verify that the Jacobi identity is satisfied
notice that, for any $i,j,k\in\{0,1,\dots,n\}$
$$
[L_{i}(f_1),[L_{j}(f_2),L_k(f_3)]]  \ = \ 0
$$
unless exactly $2$ among the indices $i,j,k$ are equal to
$n+1$. Moreover, up to change of sign one can assume that
$i=j=n+1$. In this case one verifies that
\begin{eqnarray*}
&&[L_{n+1}(f_1),[L_{n+1}(f_2),L_k(f_3)]]
=k(k-1)L_{k-2}(f_1f_2f_3)\\
&&[L_{n+1}(f_2),[L_{k}(f_3),L_{n+1}(f_1)]]
=-k(k-1)L_{k-2}(f_1f_2f_3)\\
&&[L_{k}(f_3),[L_{n+1}(f_1),L_{n+1}(f_2)]]=0
\end{eqnarray*}
and adding these identities side by side the Jacobi identity follows.
\end{proof}
\begin{definition}\label{df-heis(1,n,TRd)}{\rm
The real $*$--Lie algebra defined in Lemma \ref{Jacobi} will be denoted
$heis_\mathbb{R}(1,n,\mathcal{T})$. If $I\subset\mathbb{R}$
is a bounded Borel subsets we denote
\begin{equation}\label{df-calT-I}
\mathcal{T}_I :=\hbox{the
sub--algebra of $\mathcal{T}$ of functions with support in $I$}
\end{equation}
}\end{definition}
In analogy with the notation (\ref{gen-el-1md-alg}) we write
the generic element of $heis_\mathbb{R}(1,n,\mathcal{T})$
in the form
\begin{eqnarray}\label{gen-el-curr-alg}
\ell (\tilde f) := L_{n+1}(f_{n+1}) + \sum_{k=0}^n L_{k}(f_{k})
\qquad ;\qquad f_0,\dots,f_{n+1}\in\mathcal{T}
\end{eqnarray}
where, here and in the following, if $(f_0,\dots,f_{n+1})$
is an ordered $(n+2)$--uple of elements of $\mathcal{T}$,
we will use the notation
\begin{eqnarray}\label{df-ftilde-notat}
\tilde f :=  (f_0,\dots,f_{n+1})
\end{eqnarray}

\section{Isomorphisms between the current algebra
$heis_\mathbb{R}(1,n,\mathbb{R} \chi_I)$ and $heis_\mathbb{R}(1,n)$}
\label{subsection3.1}

In the notations of the previous section and of Definition
\ref{df-heis(1,n,TRd)}, for a bounded Borel subset $I$
of $\mathbb{R}$, we denote
$$
\chi_J(x) :=\cases{1 \ \hbox{if} \ x\in J\cr
                                     0 \ \hbox{if} \ x\notin J\cr}
$$
$$
\mathbb{R} \chi_I := \{\hbox{the real algebra of multiples of}  \ \chi_I \}
$$
Thus
\begin{eqnarray*}\label{df-MI}
heis_\mathbb{R}(1,n,\mathbb{R} \chi_I)
\subset heis_\mathbb{R}(1,n,\mathcal{H}_0(\mathbb{R}))
\end{eqnarray*}
is the $*$--Lie sub--algebra of
$heis_\mathbb{R}(1,n,\mathcal{H}_0(\mathbb{R}))$ with
linear skew--adjoint generators
\begin{eqnarray*}\label{I-gen-1mde-Heis1n}
\{  L_{k}(\chi_I) \quad : \quad k\in\{0,1,\dots,n\} \}
\end{eqnarray*}
and brackets
\begin{eqnarray}\label{[pc(chiI),qc(chiI)]b}
[ L_{n+1}(\chi_I),L_{k}(\chi_I)]=kL_{k-1}(\chi_I)
 \qquad ; \qquad k\in\{0\}\cup \{2,\dots,n\}
\end{eqnarray}
for $k\in\{2,\dots,n\}$ and the other commutators vanish. Recalling
the notation (\ref{df-q0c(f)}) one must have
\begin{eqnarray*}\label{L0(chiI)=|I|L0}
L_{0}(\chi_I)=|I|L_{0}
\end{eqnarray*}
\begin{lemma}\label{lem1}{\rm
In the notations of section \ref{Schred-repr} a real linear map \\
$\hat{s}_I:heis_\mathbb{R}(1,n,\mathbb{R} \chi_I)\rightarrow 
heis_\mathbb{R}(F,1,n)$ satisfying for some constants \\
$a_{I},b_{I},c_{k,I}\in\mathbb R^*:=\mathbb R\setminus \{0\}$
and for each $k\in\{1,\dots,n\}$
\begin{eqnarray}
\hat{s}_I(L_{0})&=&a_{I}i1\label{df-hatsI}\\
\hat{s}_I(L_{n+1}(\chi_I))&=&b_{I}ip\label{df-hatsIa}\\
\hat{s}_I(L_{k}(\chi_I))&=&c_{k,I}iq^k
\qquad ;\qquad \forall k\in\{1,\dots,n\}\label{df-hatsIb}
\end{eqnarray}
is a real $*$--Lie algebra isomorphism if and only if
\begin{eqnarray}\label{struc-ckI}
c_{k,I}= b_{I}^{-k}|I|a_{I}
\qquad ;\qquad \forall k\in\{1,\dots,n\}
\end{eqnarray}
The additional condition
\begin{eqnarray}\label{c1I=bI}
c_{1,I}= b_{I}
\end{eqnarray}
implies that $a_{I}$ must be $>0$ and:
\begin{eqnarray}\label{df-hatsIIb}
c_{k,I}=|I|^{1-\frac{k}{2}}a_{I}^{1-\frac{k}{2}}
\qquad; \qquad \forall k\in\{1,\dots,n\}
\end{eqnarray}
}\end{lemma}
\begin{remark}{\rm
In the above statement $ heis_\mathbb{R}(F,1,n)$ can be replaced by $ heis_\mathbb{R}(1,n)$
because of the real $*$--Lie algebra isomorphism between the two.
}\end{remark}
\begin{proof}
By definition $\hat{s}_I$ maps a basis of $heis_\mathbb{R}(1,n,\mathbb{R} \chi_I)$ into
a basis of  $heis_\mathbb{R}(F,1,n)$ because the constants $b_{I},c_{k,I}$
are non zero hence it defines a unique vector space
isomorphism which is a $*$--map because the constants
are real.
Moreover (\ref{df-hatsI}), (\ref{df-hatsIa}), and
(\ref{df-hatsIb}) imply that
$$
[\hat{s}_I(L_{n+1}(\chi_I),\hat{s}_I(L_{1}(\chi_I))]
=[b_{I}ip,c_{1,I}iq]
=b_{I}c_{1,I}[ip,iq]
=b_{I}c_{1,I}i 1
$$
while (\ref{[pc(chiI),qc(chiI)]b}) and (\ref{df-hatsIb}) imply that
$$
\hat{s}_I([L_{n+1}(\chi_I),L_{1}(\chi_I)])
=\hat{s}_I(|I|L_{0})
=|I|\hat{s}_I(L_{0})
=|I|a_{I}i1
$$
The isomorphism condition then implies that
\begin{eqnarray}\label{bIc1I=|I|}
b_{I}c_{1,I}  =|I|a_{I}
\end{eqnarray}
The same argument, using (\ref{[pc(chiI),qc(chiI)]b}), shows that
for all $k\in\{2,\dots,n\}$
$$
[\hat{s}_I(L_{n+1}(\chi_I)),\hat{s}_I(L_{k}(\chi_I))]
=[b_{I}ip,c_{k,I}iq^k]
=b_{I}c_{k,I}[ip,iq^k]
=b_{I}c_{k,I}kiq^{k-1}
$$
$$
\hat{s}_I([L_{n+1}(\chi_I),L_{k}(\chi_I)])
=\hat{s}_I(kL_{k-1}(\chi_I))
=k\hat{s}_I(L_{k-1}(\chi_I))
=kc_{k-1,I}iq^{k-1}
$$
and the isomorphism condition implies that
$$
b_{I}c_{k,I}=c_{k-1,I}
\Leftrightarrow c_{k,I}=b_{I}^{-1}c_{k-1,I}
=b_{I}^{-2}c_{k-2,I}=\dots = b_{I}^{-(k-1)}c_{1,I}
= b_{I}^{-k}|I|a_{I}
$$
which is (\ref{struc-ckI}). Finally, if (\ref{c1I=bI}) holds,
then (\ref{bIc1I=|I|}) becomes
$$
b_{I}^2  =|I|a_{I}
$$
Thus $a_{I}$ must be $>0$ and $b_{I}  =|I|^{1/2}a_{I}^{1/2}$
which implies (\ref{df-hatsIIb}).
\end{proof} 

\begin{remark}{\rm
In the following we fix condition (\ref{c1I=bI}) and put
\begin{eqnarray}\label{aI=1}
a_{I}=1 
\end{eqnarray}
for all $I$ so that the real $*$--Lie algebras
isomorphism  $\hat{s}_I$ is given by (\ref{df-hatsI})
and (\ref{df-hatsIIb}). Therefore its inverse
$\hat{s}^{-1}_I$ is given, on the generators, by:
\begin{eqnarray*}
\hat{s}^{-1}_I(i1)&=& L_{0}\\
\hat{s}^{-1}_I(ip)&=&|I|^{-\frac{1}{2}}L_{n+1}(\chi_I)\\
\hat{s}^{-1}_I(iq^k)&=&|I|^{k/2-1}L_{k}(\chi_I)
\qquad;\qquad \forall k\in\{1,\dots,n\}
\end{eqnarray*}
The reason why the additional conditions (\ref{struc-ckI}) and
(\ref{aI=1}) are necessary will be explained in Remark \ref{iso-and-proj}
at the end of section \ref{Exist-fact-st}.
}\end{remark}

\begin{remark}
Lemma \ref{lem1} and condition (\ref{aI=1}) mean that,
for any bounded Borel set $I\subset\mathbb{R}$,
$heis_\mathbb{R}(1,n,\mathbb{R} \chi_I)$
can be  identified to a copy of $ heis_\mathbb{R}(1,n) $ with
the rescaled basis
\begin{eqnarray}\label{I-scal-gen-Heis1n}
\{i|I|L_{0}\quad,\quad i|I|^{\frac{1}{2}}L_{n+1}\quad,\quad
i|I|^{1-\frac{k}{2}}L_{k}\quad,\quad k=1,\dots,n\}
\end{eqnarray}
\end{remark}

In analogy with (\ref{gen-el-1md-alg}), we parametrize the elements of
$heis_\mathbb{R}(1,n,\mathbb{R}\chi_I)$, with elements
of $\mathbb{R}\times\mathbb{R}_n[X]$, and we write
\begin{eqnarray}\label{LI}
\ell_I(u,P):=uL_{n+1}(\chi_I)+P(L(\chi_I))
\qquad; \qquad u\in\mathbb{R}
\end{eqnarray}
where $P:=\sum_{j=0}^na_jX^j$ is a polynomial in
one indeterminate and we use the convention
\begin{eqnarray}\label{LIc}
P(L(\chi_I)):=\sum_{j=0}^na_jL_{j}(\chi_I)
:=a_0|I|L_{0}+\sum_{j=1}^na_jL_{j}(\chi_I)
\end{eqnarray}
The image of such an element under the
isomorphism $\hat{s}_I$ is
\begin{eqnarray}\label{df-hat-sI(lI)}
\hat{s}_I(\ell_I(u,P))= i(u|I|^{\frac{1}{2}}p+P_I(q))
\end{eqnarray}
where by definition:
\begin{eqnarray}\label{df-PI}
P_I(X):=\sum_{j=0}^na_j|I|^{1-\frac{j}{2}}X^j
=a_0|I|1+\sum_{j=1}^n a_j|I|^{1-\frac{j}{2}}X^j
\end{eqnarray}
Introducing the linear change of coordinates in
$\mathbb{R}\times\mathbb{R}_n[X]$ defined by
\begin{eqnarray}\label{df-hat-kI}
\hat{k}_I(u,P):=(u|I|^{\frac{1}{2}},P_I)
\equiv \left(u|I|^{\frac{1}{2}}, (a_j|I|^{1-\frac{j}{2}}) \right)
\end{eqnarray}
where $P_I$ is defined by (\ref{df-PI}) we see that, in the
notations (\ref{gen-el-1md-alg}) and (\ref{LI})
one has
\begin{eqnarray}\label{hatsI-lI=l0hatkI}
\hat{s}_I\circ \ell_I=  \ell_0\circ \hat{k}_I
\end{eqnarray}

\section{The group $\hbox{Heis}(1,n,\mathbb{R}\chi_I)$ and its $C^*$--algebra}

In the notations and assumptions of section \ref{subsection3.1}
we have seen that $heis_\mathbb{R}(1,n,\mathbb{R}\chi_I)$
is isomorphic to $heis_\mathbb{R}(1,n)$. Since
$\mathbb{R}^{n+2}$ is connected and simply connected, the Lie group
of $heis_\mathbb{R}(1,n,\mathbb{R}\chi_I)$, denoted
$\hbox{Heis}(1,n,\mathbb{R}\chi_I)$ is isomorphic to
$\hbox{Heis}(1,n)$.
In analogy with the notation (\ref{el-Heis-gr}), the generic element of
$\hbox{Heis}(1,n,\mathbb{R}\chi_I)$ will be denoted
\begin{eqnarray}\label{el-Heis-I-gr}
e^{\ell_I(u,P)} \qquad;\qquad
(u,P)\in\mathbb{R}\times\mathbb{R}_n[X]
\end{eqnarray}
\begin{definition}\label{W1nI}{\rm
For any bounded Borel set $I\subset \mathbb{R}$ we denote
$$
\mathcal{W}^0_{1,n;I} := \mathcal{W}(\hbox{Heis}(1,n,\mathbb{R}\chi_I))
$$
the free group--$C^*$--algebra of the group
$\hbox{Heis}(1,n,\mathbb{R}\chi_I))$.
In analogy with (\ref{gen-el-Heis}), its generators will be
called the \textit{one mode $n$--th degree Weyl operators localized
on $I$} and denoted
\begin{eqnarray}\label{gen-el-Heis-I-alg}
W^0_I(u,P):= W_{e^{\ell_I(u,P)}}\in\mathcal{W}^0_{1,n;I}
\end{eqnarray}
}\end{definition}
\begin{remark}{\rm
Since the groups $\hbox{Heis}(1,n,\mathbb{R}\chi_I))$ and
$\hbox{Heis}(1,n)$ are isomorphic, the same is true for the corresponding
free group--$C^*$--algebras.
}\end{remark}

In the following section we show that, in these $C^*$--algebra isomorphisms,
the group generators of $\mathcal{W}^0_{1,n;I}$ are mapped into a set of
group generators of $\mathcal{W}^0_{1,n}$ which depends on $I$
and we introduce a construction that allows to get rid of this dependence.

\subsection{  $C^*$--algebras isomorphism}\label{C*-alg-iso}

In the notations (\ref{gen-el-1md-alg}) and (\ref{el-Heis-I-gr})
the map
\begin{eqnarray*}\label{Schr-rep-Heis-I}
e^{\ell_I(u,P)}\in  \hbox{Heis}(1,n,\mathbb{R}\chi_I)
 \ \mapsto  \ e^{\hat{s}_I(\ell_I(u,P))}\in\hbox{Heis}(1,n)
\end{eqnarray*}
where $\hat{s}_I$ the isomorphism defined in Lemma \ref{lem1},
is a Lie group isomorphism, hence it
can be extended to a $C^*$--isomorphism of the corresponding free
group--$C^*$--algebras.\\
This extension will be denoted with the symbol:
$$
s^0_I: \mathcal{W}^0_{1,n;I}\to \mathcal{W}^0_{1,n}
$$
In view of the identity (\ref{hatsI-lI=l0hatkI}), and in the
notations (\ref{gen-el-Heis}) and (\ref{gen-el-Heis-I-alg}),
the explicit form of $s^0_I$ on the generators is given by
\begin{eqnarray}\label{expl-frm-WI(u,P)}
s^0_I(W^0_I(u,P))=W^0(\hat{k}_I(u,P))
\end{eqnarray} 
where $\hat{k}_I$ is the linear map defined by (\ref{df-hat-kI}) and
$(u,P)\in\mathbb{R}\times\mathbb{R}_n[X]$. \\
It is clear from (\ref{df-hat-kI}) and (\ref{expl-frm-WI(u,P)}) that,
as a vector space, $s^0_I(\mathcal{W}^0_{1,n;I})$
coincides with $\mathcal{W}^0_{1,n}$.
In this section we will prove that the map
\begin{equation}\label{df-kI}
W^0(u,P) \in \mathcal{W}^0_{1,n} \mapsto W^0(\hat{k}_I(u,P))
\in\mathcal{W}^0_{1,n}
\end{equation}
induces a $C^*$--algebra automorphism denoted $ k_I$.
To this goal we use
$$
W^0(\hat{k}_I(u,P))W^0(\hat{k}_I(v,Q))
=W^0(\hat{k}_I(u,P)\circ \hat{k}_I(v,Q))
$$
and the following result.
\begin{lemma}\label{law-composition}{\rm
For all $u\in\mathbb{R}$ and $P\in\mathbb{R}_n[X]$, let $\hat{k}_I$ 
be the linear map defined by (\ref{df-hat-kI}). Then,
denoting with the same symbol $\hat{k}_I$ its restriction
on $\mathbb{R}_n[X]$, one has:
\begin{eqnarray*}
\hat{k}_I\circ T_u(P)&=&T_{u|I|^{\frac{1}{2}}}\circ
\hat{k}_I(P)\\
\hat{k}_I^{-1}\circ T_u^{-1}(P)&=&
T_{u|I|^{-\frac{1}{2}}}^{-1}\circ \hat{k}_I^{-1}(P)\\
\hat{k}_I\circ T_u^{-1}(P)&=&T_{u|I|^{\frac{1}{2}}}^{-1}
\circ \hat{k}_I(P)\\
\hat{k}_I^{-1}\circ T_u(P)&=&T_{u|I|^{-\frac{1}{2}}}
\circ \hat{k}_I^{-1}(P)
\end{eqnarray*}
}\end{lemma}
\begin{proof}
Since both $T_u$ and $\hat{k}_I$ are linear maps, it
is sufficient to prove the lemma for
$P(X)=X^k$ ($k\in\{0,\dots,n\}$).
For $k=0$ all the identities in the lemma are  obviously true.
Let $k\in\{1,\dots,n\}$. Then from the identity (\ref{aa})
one has
\begin{eqnarray}\label{id1}
T_{u|I|^{\frac{1}{2}}}\circ \hat{k}_I(X^k)&=&
T_{u|I|^{\frac{1}{2}}}(|I|^{1-\frac{k}{2}}X^k)\nonumber\\
&=&|I|^{1-\frac{k}{2}}\,T_{u|I|^{\frac{1}{2}}}(X^k)\nonumber\\
&=&|I|^{1-\frac{k}{2}}\Big[\sum_{h=0}^{k-1}
\frac{k!}{(k+1-h)!h!}u^{k-h}|I|^{\frac{k-h}{2}}X^h+X^k\Big]
\nonumber\\
&=&\sum_{h=0}^{k-1}\frac{k!}{(k+1-h)!h!}
u^{k-h}|I|^{1-\frac{h}{2}}X^h+|I|^{1-\frac{k}{2}}\,X^k\nonumber\\
&=&\hat{k}_I\circ T_u(X^k)
\end{eqnarray}
(\ref{id1}) is equivalent to
$$
T_{u|I|^{\frac{1}{2}}}\circ \hat{k}_I=\hat{k}_I\circ T_u
\Leftrightarrow
\hat{k}_I^{-1}\circ T^{-1}_{u|I|^{\frac{1}{2}}}
=T^{-1}_u\circ\hat{k}_I^{-1}
$$
Replacing $u$ by $u|I|^{-\frac{1}{2}}$, this yields
\begin{eqnarray}\label{id2}
\hat{k}_I^{-1}\circ T_u^{-1}
=T_{u|I|^{-\frac{1}{2}}}^{-1}\circ \hat{k}_I^{-1}
\end{eqnarray}
>From identities (\ref{id1}) and (\ref{id2}), one gets
\begin{eqnarray*}
T_{u|I|^{\frac{1}{2}}}\circ\hat{k}_I\circ
T_u^{-1}&=&\hat{k}_I\\
T_{u|I|^{-\frac{1}{2}}}\circ\hat{k}_I^{-1}\circ
T_u^{-1}&=&\hat{k}_I^{-1}
\end{eqnarray*}
or equivalently
\begin{eqnarray*}
\hat{k}_I\circ T_u^{-1}&=&
T_{u|I|^{\frac{1}{2}}}^{-1}\circ \hat{k}_I\\
\hat{k}_I^{-1}\circ T_u&=&T_{u|I|^{-\frac{1}{2}}}\circ \hat{k}_I^{-1}
\end{eqnarray*}
\end{proof}
\begin{proposition}\label{prop2}{\rm
$\hat{k}_I$ is a group automorphism for the composition law
(\ref{1Npol-comp-law1}).
}\end{proposition}

\begin{proof}
We have to prove that
for all $(u,P),\,(v,Q)\in\mathbb{R}\times\mathbb{R}_n[X]$,
one has
$$
(\hat{k}_I(u,P)\circ \hat{k}_I(v,Q))
=\hat{k}_I\Big((u+v);T_{(u+v)}^{-1}\big(T_{u}P+T_{v}S_uQ\big)\Big)
$$
We know that
$$
\hat{k}_I(u,P)\circ \hat{k}_I(v,Q)
=(u|I|^{\frac{1}{2}},P_I)\circ(v|I|^{\frac{1}{2}},Q_I)
$$
where $P_I(X)=P(|I|^{-\frac{1}{2}}X)$ and
$Q_I(X) =Q(|I|^{-\frac{1}{2}}X)$. But from
(\ref{comp-law-Heis1n}) we know that
\begin{eqnarray*}\label{k-1}
(u|I|^{\frac{1}{2}},P_I)\circ(v|I|^{\frac{1}{2}},Q_I)\!&=&\!\Big((u+v)|I|^{\frac{1}{2}} \ , \  T_{(u+v)|I|^{\frac{1}{2}}}^{-1}\big(T_{u|I|^{\frac{1}{2}}}
P_I+T_{v|I|^{\frac{1}{2}}}S_uQ_I\big)\Big)\\
&\!\!\!\!\!\!=&\!\!\!\!\!\Big((u+v)|I|^{\frac{1}{2}} ,  T_{(u+v)|I|^{\frac{1}{2}}}^{-1}
\big(T_{u|I|^{\frac{1}{2}}}\hat k_I(P)
+T_{v|I|^{\frac{1}{2}}}S_u\hat k_I(Q)\big)\Big)
\end{eqnarray*}
Furthermore, from Lemma \ref{law-composition}, we know that
$$
T_{(u+v)|I|^{\frac{1}{2}}}^{-1}T_{u|I|^{\frac{1}{2}}}\hat k_I(P)
=\hat k_IT_{(u+v)}^{-1}T_{u}(P)
$$
Moreover, using
$$
S_u\hat k_I(Q) =\hat k_IS_u(Q)
$$
We also have
\begin{eqnarray*}
T_{(u+v)|I|^{\frac{1}{2}}}^{-1}T_{v|I|^{\frac{1}{2}}}S_u\hat{k}_I(Q)&=&T_{(u+v)|I|^{\frac{1}{2}}}^{-1}T_{v|I|^{\frac{1}{2}}}\hat k_IS_u(Q)\\
&=&T_{(u+v)|I|^{\frac{1}{2}}}^{-1}\hat k_IT_{v}S_u(Q)\\
&=&\hat k_IT_{(u+v)}^{-1}T_{v}S_u(Q)
\end{eqnarray*}
Hence, one gets
\begin{eqnarray*}
(u|I|^{\frac{1}{2}},P_I)\circ(v|I|^{\frac{1}{2}},Q_I)
&=&\big((u+v)|I|^{\frac{1}{2}} \ , \ \hat k_IT_{(u+v)}^{-1}\big(T_{u}(P)+T_{v}S_u(Q)\big)\big)\\
&=&\hat k_I\Big((u+v) \ , \  T_{(u+v)}^{-1}\big(T_{u}(P)+T_{v}S_u(Q)\big)\Big)\\
&=&\hat k_I\big((u,P)\circ(v,Q)\big)
\end{eqnarray*}
and this proves the statement.
\end{proof}

\begin{corollary}{\rm
The map:
\begin{eqnarray*}\label{df-sI2}
s_I:= k_I^{-1}\circ s^0_I \ : \ \mathcal{W}^0_{1,n;I}\to \mathcal{W}^0_{1,n}
\end{eqnarray*}
is a $C^*$--algebra isomorphism characterized by the condition
\begin{eqnarray}\label{df-sI2}
s_I (W_I^0(u,P)) = W^0(u,P)
 \qquad ; \qquad \forall (u,P)\in \mathbb R^{n+2}
\end{eqnarray}
}\end{corollary}
\begin{proof}
(\ref{df-sI2}) is clear from (\ref{expl-frm-WI(u,P)}) and the definition 
(\ref{df-kI}) of $k_I$.\\
We know that $s^0_I$ is a $C^*$--algebra isomorphism. 
>From Proposition \ref{prop2}
we know that $\hat k_I$ is a group automorphism for the composition 
law defined by (\ref{1Npol-comp-law1}).
Because of the linear independence of the free group algebra
generators $k_I$ extends to a $C^*$--algebra automorphism.
Thus $s_I$ is composed of an isomorphism with an automorphism and
the thesis follows.
\end{proof}

\section{The inductive limit}\label{The-ind-lim}

In the following, when speaking of tensor products of $C^*$--algebras,
it will be understood that a choice of a cross norm has been fixed and that all
tensor products are referred to the same choice.


For a bounded Borel subset $I$ of $\mathbb{R}$, let $\mathcal{W}^0_{1,n;I}$ be
the $C^*$--algebra in Definition \ref{W1nI} and let the isomorphisms
$s_I : \mathcal{W}^0_{1,n;I}\to \mathcal{W}^0_{1,n}$
defined by (\ref{df-sI2}). For $\pi=(I_j)_{j\in F}\in Part_{fin}(I)$ define the $C^*$--algebra
\begin{eqnarray}
\mathcal{W}^0_{1,n;I;\pi} := \bigotimes_{j\in F} \mathcal{W}^0_{1,n;I_j}\label{df-W0(I;pi)}
\end{eqnarray}
the injective $C^*$--homomorphism  ($C^*$--embedding)
\begin{eqnarray}
z_{I,\pi}:=\left(\bigotimes^{(diag)}_{j\in F}s_{I_j}^{-1}\right)
\circ s_{I} \ : \ \mathcal{W}^0_{1,n;I}\rightarrow
\bigotimes_{j\in F}\mathcal{W}^0_{1,n;I_j} =\mathcal{W}^0_{1,n;I;\pi}\label{df-z(I;pi)} \nonumber
\end{eqnarray}
Then, for any $\pi\prec  \pi'\in Part_{fin}(I)$, the map
\begin{eqnarray}\label{df-z(I;pi',pi)} 						 z_{I;\pi,\pi'}:=\bigotimes_{j\in F}
\left(\bigotimes^{(diag)}_{I_j\supseteq I'\in \pi'}s_{I'}^{-1}\right)
\circ s_{I_j} \ : \ \mathcal{W}^0_{1,n;I;\pi}
\rightarrow\bigotimes_{j\in F}
\bigotimes_{I_j\supseteq I'\in \pi'}\mathcal{W}^0_{1,n;I'}
=  \mathcal{W}^0_{1,n;I;\pi'}
\end{eqnarray}
is a $C^*$--embedding. Moreover, by construction and in the notations of Definition \ref{W1nI}, for all $u\in\mathbb{R}$ and $P=\sum_{j=0}^na_jX^n\in\mathbb{R}_n[X]$, one has
\begin{eqnarray*}\label{df-z(I;pi)-1} 						z_{I;\pi,\pi'} z_{I,\pi}(W^0_I(u,P)):=\bigotimes_{I'\in \pi'}W^0_{I'}(u,P)
=z_{I,\pi'}(W^0_I(u,P))
\in\mathcal{W}^0_{1,n;I;\pi'}
\end{eqnarray*}
\begin{lemma}\label{projectiv-I}{\rm
The family 
\begin{eqnarray}\label{df-I-pi-ind-syst} 				
\left\{ (\mathcal{W}^0_{1,n;I;\pi})_{\pi\in Part_{fin}(I)} \ , \
(z_{I;\pi,\pi'})_{\pi\prec \pi'\in Part_{fin}(I)}\right\}
\end{eqnarray}
 is an inductive system of $C^*$--algebras, i.e. for all $\pi\prec\pi'$, 
$z_{I;\pi,\pi'}$ is a morphism and if
$\pi\prec \pi'\prec \pi''\in Part_{fin}(I)$ one has
\begin{equation}\label{zIpp'-ind-syst}
z_{I;\pi',\pi''}z_{I;\pi,\pi'}=z_{I;\pi,\pi''}
\end{equation}
}\end{lemma}
\begin{proof}
We have already proved that the $z_{I,\pi,\pi'}$ are $C^*$--embeddings.
Therefore it remains to prove (\ref{zIpp'-ind-syst}). To this goal, for
$\pi,\pi',\pi''$ as in the statement, using the identity
$$
\bigotimes_{I'\in \pi'} \ = \ \bigotimes_{I\in \pi}
\bigotimes_{\pi'\ni I'\subseteq I}
$$
one finds
$$
 z_{I;\pi'',\pi'} z_{I;\pi,\pi'}
=\left(\bigotimes_{I'\in \pi'}\left(\bigotimes^{(diag)}_{\pi''\ni I''\subseteq I'}
s_{I''}^{-1}\right)\circ s_{I'}\right)
\left(\bigotimes_{I\in \pi}\left(
\bigotimes^{(diag)}_{\pi'\ni I'\subseteq I}s_{I'}^{-1}\right)\circ s_{I}\right)
$$
$$
 \ = \ \left(\bigotimes_{I\in \pi}\bigotimes_{\pi'\ni I'\subseteq I}
\left(\bigotimes^{(diag)}_{\pi''\ni I''\subseteq I'}
s_{I''}^{-1}\right)\circ s_{I'}\right)
\left(\bigotimes_{I\in \pi}\left(
\bigotimes^{(diag)}_{\pi'\ni I'\subseteq I}s_{I'}^{-1}\right)\circ s_{I}\right)
$$
$$
 \ = \ \bigotimes_{I\in \pi}\left(\bigotimes_{\pi'\ni I'\subseteq I}
\left(\bigotimes^{(diag)}_{\pi''\ni I''\subseteq I'}
s_{I''}^{-1}\right)\circ s_{I'}\right)
\left(\left(
\bigotimes^{(diag)}_{\pi'\ni I'\subseteq I}s_{I'}^{-1}\right)\circ s_{I}\right)
$$
$$
 \ = \ \bigotimes_{I\in \pi}\bigotimes_{\pi'\ni I'\subseteq I}\left(
\left(\bigotimes^{(diag)}_{\pi''\ni I''\subseteq I'}
s_{I''}^{-1}\right)\circ s_{I'}
\left(s_{I'}^{-1}\right)\circ s_{I}\right)
$$
$$
 \ = \ \bigotimes_{I\in \pi}\left(
\left(\bigotimes_{\pi'\ni I'\subseteq I}
\bigotimes^{(diag)}_{\pi''\ni I''\subseteq I'}
s_{I''}^{-1}\right)\circ s_{I}\right)
$$
$$
 \ = \ \bigotimes_{I\in \pi}\left(
\left(\bigotimes^{(diag)}_{\pi''\ni I''\subseteq I}
s_{I''}^{-1}\right)\circ s_{I}\right)
= z_{I;\pi'',\pi} 
$$
\end{proof}

\begin{definition}\label{df-ind-lim}{\rm
For any bounded Borel subset $I$ of $\mathbb{R}$, we denote
$$
\left\{ \mathcal{W}_{1,n;I},
(\tilde z_{I;\pi})_{\pi\in Part_{fin}(I)}\right\}
$$
the inductive limit of the family (\ref{df-I-pi-ind-syst}) i.e.,
$ \mathcal{W}_{1,n;I}$ is a $C^*$--algebra and
for any $\pi\in Part_{fin}(I)$ and in the notation (\ref{df-W0(I;pi)}),
$$
\tilde z_{I;\pi}: \mathcal{W}^0_{1,n;I;\pi} \to \mathcal{W}_{1,n;I}
$$
is an embedding satisfying
\begin{eqnarray*}\label{fact-fam-W1nI}
\tilde z_{I;\pi'}z_{I;\pi,\pi'}=\tilde z_{I;\pi}
\qquad;\qquad \forall \pi\prec \pi'\in Part_{fin}(I)
\end{eqnarray*}
}\end{definition}
{\bf Remark}
Intuitively one can think of the elements of $\mathcal{W}_{1,n;I}$ 
as a realization of the {\it non--linear Weyl operators}:
(\ref{form-expr-NLW}) with finitely valued, compact support, 
test functions.

\subsection{Factorizable families of $C^*$--algebras }

\begin{definition}\label{cont-tens-prod-prop}
A family of $C^*$--algebras $\{\mathcal{W}_I\}$,
indexed by the bounded Borel subsets of $\mathbb{R}$,
is called \textit{factorizable} if, for every bounded Borel $I\subset\mathbb{R}$ and
every Borel partition $\pi$ of $I$, there is an isomorphism
\begin{eqnarray*}\label{df-cont-tens-prod}
u_{I,\pi}: \bigotimes_{I_j\in \pi} \mathcal{W}_{I_j}\rightarrow
\mathcal{W}_I
\end{eqnarray*}
 If this is the case, an operator $w_I\in\mathcal{W}_I$
is called factorizable if there exist operators $w_{I_j}\in\mathcal{W}_{I_j}$
($I_j\in \pi$) such that
\begin{eqnarray}\label{df-wI-fact}
u_{I,\pi}^{-1}(w_I) = \bigotimes_{I_j\in \pi} w_{I_j}
\end{eqnarray}
\end{definition}

\begin{remark}
In the following, for a given bounded Borel set $I$, when $\pi\equiv \{I\}$
is the partition of $I$, consisting of the only set $I$, we will use the notation
\begin{eqnarray*}\label{df-tilde-zI}
\tilde z_{I}:=\tilde z_{I;\{I\}} \ : \
\mathcal{W}^0_{1,n;I} \to  \mathcal{W}_{1,n;I}
\end{eqnarray*}
\end{remark}

We want to prove that:
\begin{enumerate}
\item[(i)] the family of $C^*$--algebras
\begin{eqnarray}\label{fact-fam-W1nI}
\left\{\mathcal{W}_{1,n;I} \ : \ I\hbox{--bounded Borel subset of }
 \ \mathbb{R} \right\}
\end{eqnarray}
where the algebras $\mathcal{W}_{1,n;I}$ are those introduced
in Definition \ref{df-ind-lim}, is factorizable in the sense of
Definition \ref{cont-tens-prod-prop};
\item[(ii)] for any bounded Borel set $I$, the operators
\begin{eqnarray}\label{fact-NL-Weyl}
W_I(u,P):=\tilde z_{I}(W^0_I(u,P)) \in \mathcal{W}_{1,n;I}
\qquad;\qquad  W^0_I(u,P)\in\mathcal{W}^0_{1,n;I}
\end{eqnarray}
are factorizable in the sense of (\ref{df-wI-fact}).
\end{enumerate}

To this goal let us remark that, if $I, J$ are
disjoint bounded Borel sets in $\mathbb{R}$, then the map
\begin{eqnarray}\label{piIpiJ=piIunJ}
\!\!\!\!\!\!\!\!\!\!(\pi_{I},\pi_{J})\in Part_{fin}(I)\times Part_{fin}(J) \ \mapsto \
\pi_{I\cup J}:=\{\pi_{I}\cup\pi_{J}\}\in Part_{fin}(I\cup J)
\end{eqnarray}
defines a canonical bijection between $Part_{fin}(I)\times
Part_{fin}(J)$ and\\ $Part_{fin}(I\cup J)$ such that, if
$\pi_{I}\prec \pi'_{I}\in Part_{fin}(I)$ and $\pi_{J}\prec
\pi'_{J}\in Part_{fin}(J)$, then $\pi_{I\cup J}\prec \pi'_{I\cup
J}\in Part_{fin}(I\cup J)$.
\begin{lemma}\label{lm-factoriz}
Let $I, J$ be disjoint bounded Borel sets in $\mathbb R $.
Then the inductive system of $C^*$--algebras
\begin{eqnarray}\label{df-IunJ-pi-ind-syst} 				
\left\{ (\mathcal{W}^0_{1,n;I\cup J;\pi_{I\cup J}})
_{\pi_{I\cup J}\in Part_{fin}(I\cup J)} \ , \
(z_{I\cup J;\pi_{I\cup J},\pi'_{I\cup J}})
_{\pi_{I\cup J}\prec \pi'_{I\cup J}\in Part_{fin}(I\cup J)}\right\}
\end{eqnarray}
is isomorphic to the inductive system of $C^*$--algebras
$$
\left\{ (\mathcal{W}^0_{1,n;I;\pi_{I}}\otimes\mathcal{W}^0_{1,n;J;\pi_{J}})
_{(\pi_{I},\pi_{J})\in Part_{fin}(I)\times Part_{fin}(J)} \ , \ \right.
$$
\begin{eqnarray}\label{df-ItensJ-pi-ind-syst} 				
\left. (z_{I;\pi_{I},\pi'_{I}}\otimes z_{J;\pi_{J},\pi'_{J}})
_{\pi_{I}\prec \pi'_{I}\in Part_{fin}(I), \pi_{J}\prec \pi'_{J}
\in Part_{fin}(J)}\right\}
\end{eqnarray}
in the sense that, for each $\pi_{I}\in Part_{fin}(I)$ and  
$\pi_{J}\in Part_{fin}(J)$, then 
there exists a $C^*$--algebra isomorphism 
$$
u_{I,J,\pi_{I},\pi_{J}} \ : \ 
\mathcal{W}^0_{1,n;I;\pi_{I}}\otimes\mathcal{W}^0_{1,n;J;\pi_{J}}
\to \mathcal{W}^0_{1,n;I\cup J;\pi_{I\cup J}}
$$
such that, for each $\pi_{I}\prec \pi'_{I}\in Part_{fin}(I)$ and  
$\pi_{J}\prec \pi'_{J}\in Part_{fin}(J)$, one has in the notation
(\ref{piIpiJ=piIunJ} )
\begin{eqnarray}\label{intertw-prop-uIJpiIpiJ} 				
u_{I,J,\pi_{I},\pi_{J}}\circ 
(z_{I;\pi_I,\pi_I'}\otimes z_{J;\pi_J,\pi_J'})
=z_{I\cup J;\pi_{I\cup J},\pi_{I\cup J}'}
\end{eqnarray}
\end{lemma}
\begin{proof}
With the notations above, from (\ref{df-W0(I;pi)}) one deduces that
\begin{eqnarray}
\mathcal{W}^0_{1,n;I;\pi_I}\otimes\mathcal{W}^0_{1,n;J;\pi_J} &:=&
(\bigotimes_{I'\in \pi_I} \mathcal{W}^0_{1,n;I'})\otimes
(\bigotimes_{J'\in \pi_J} \mathcal{W}^0_{1,n;J'})\nonumber\\
&\equiv& \bigotimes_{K\in \pi_{I\cup J}} \mathcal{W}^0_{1,n;K}=\mathcal{W}^0_{1,n;I\cup J;\pi_{I\cup J}}\label{iso-W(ItensJ;pi)-W(IunJ;pi)}
\end{eqnarray}
Denote 
$$
u_{I\otimes J,I\cup J}:
\mathcal{W}^0_{1,n;I;\pi_I}\otimes\mathcal{W}^0_{1,n;J;\pi_J}
\to \mathcal{W}^0_{1,n;I\cup J;\pi_{I\cup J}}
$$
the isomorphism defined by (\ref{iso-W(ItensJ;pi)-W(IunJ;pi)}). If $\pi_{I}\prec \pi'_{I}\in Part_{fin}(I)$ and  
$\pi_{J}\prec \pi'_{J}\in Part_{fin}(J)$, then clearly $\pi_{I\cup J}\prec \pi'_{I\cup J}\in Part_{fin}(I\cup J)$
and from (\ref{df-z(I;pi',pi)}) we see that
\begin{eqnarray*}
&&u_{I,J,\pi_{I},\pi_{J}}\circ(z_{I;\pi_{I},\pi'_{I}}\otimes z_{J;\pi_{J},\pi'_{J}})\\
&=&u_{I,J,\pi_{I},\pi_{J}}\circ\left(\left(\bigotimes_{I_j\in \pi_{I}}\left(\bigotimes^{(diag)}_{I_j\supseteq I'\in \pi'_{I}}s_{I'}^{-1}\right)
\circ s_{I_j} \right)\otimes 
\left(\bigotimes_{J_h\in \pi_{J}}
\left(\bigotimes^{(diag)}_{J_h\supseteq J'\in \pi'_{J}}s_{J'}^{-1}\right)
\circ s_{J_h} \right)\right)\\
&=&\bigotimes_{H_l\in \pi_{I\cup J}}\left(\bigotimes^{(diag)}_{H_l\supseteq K\in \pi'_{I\cup J}}s_{K}^{-1}\right)
\circ s_{H_l} =z_{I\cup J;\pi_{I\cup J},\pi'_{I\cup J}}
\end{eqnarray*}
which proves (\ref{intertw-prop-uIJpiIpiJ}).
\end{proof}
\begin{theorem}\label{thm-factoriz}
\begin{enumerate}
\item[(i)] The family of $C^*$--algebras defined by (\ref{fact-fam-W1nI})
is factorizable.
\item[(ii)] The operators defined by (\ref{fact-NL-Weyl}) are factorizable.
\end{enumerate} 
\end{theorem}
\begin{proof}
We apply Definition \ref{cont-tens-prod-prop} to the case in which
the family $\mathcal F$ is the family of bounded Borel sets in 
$\mathbb R$. By induction it will be sufficient to prove that, 
if $I, J$ are disjoint bounded Borel sets in $\mathbb R$, then 
there exists a $C^*$--algebra isomorphism 
$$
u_{I,J} \ : \ 
\mathcal{W}_{1,n;I}\otimes\mathcal{W}_{1,n;J}
\to \mathcal{W}_{1,n;I\cup J}
$$
Since $\mathcal{W}_{1,n;I}\otimes\mathcal{W}_{1,n;J}$
is the inductive limit of the system (\ref{df-ItensJ-pi-ind-syst})
and $\mathcal{W}_{1,n;I\cup J}$ is the inductive limit of the 
system (\ref{df-IunJ-pi-ind-syst}), the statement follows from
Lemma \ref{lm-factoriz} because isomorphic inductive systems
have isomorphic inductive limits.

The factorizability of the operators (\ref{fact-NL-Weyl}) follows from
the identity (\ref{fact-fam-W1nI}).
\end{proof}

>From Theorem \ref{thm-factoriz} it follows that, if $I\subset J$
are bounded Borel sets in $\mathbb R$, then the map 
\begin{eqnarray}\label{df-jI-emb} 							
j_{I;J} \ : \ w_I\in \mathcal{W}_{1,n;I}\to
 w_I\otimes 1_{J\setminus I}\in\mathcal{W}_{1,n;J}
\end{eqnarray}
is a $C^*$--algebra isomorphism. Since clearly, for $I\subset J\subset K$
bounded Borel sets in $\mathbb R$, 
$1_{J\setminus I}\otimes 1_{K\setminus J}\equiv 1_{K\setminus I}$,
it follows that 
\begin{eqnarray}\label{df-jI-ind-syst} 							
\left\{(\mathcal{W}_{1,n;I}) \ , \ (j_{I;J})
 \ , \ I\subset J\in \hbox{bounded Borel sets in} \ \mathbb R
\right\}
\end{eqnarray}
is an inductive system of $C^*$--algebras. 

\begin{definition}
The inductive limit of the system (\ref{df-jI-ind-syst}) will be denoted
\begin{eqnarray*}\label{df-I-ind-lim}
\left\{\mathcal{W}_{1,n;\mathbb{R}} \ , \ (j_{I})
 \ , \ I\in \hbox{bounded Borel sets in} \ \mathbb{R}
\right\}
\end{eqnarray*}
\end{definition}
Since the $j_{I}:\mathcal{W}_{1,n;I}\to \mathcal{W}_{1,n;\mathbb{R}}$
are injective embeddings, the family $(j_{I}(\mathcal{W}_{1,n;I}))$
is factorizable and one can introduce the more intuitive notation:
$$
j_{I}(\mathcal{W}_{1,n;I})\equiv \mathcal{W}_{1,n;I}\otimes 1_{I^c}
$$

\section{Existence of factorizable states on $\mathcal{W}_{1,n;\mathbb{R}}$}
\label{Exist-fact-st}

In the notation (\ref{gen-el-Heis-I-alg}) and with the operators $W_I(u,P)$
defined by (\ref{fact-NL-Weyl}), using factorizability of the family
$(\mathcal{W}_{1,n;I})$ and of the corresponding generators,
for any $I\subset \mathbb{R}$ bounded Borel and any finite partition $\pi$ of $I$,
we will use the identifications
\begin{eqnarray}
\mathcal{W}_{1,n;I}&\equiv&
j_{I}(\mathcal{W}_{1,n;I})\equiv \mathcal{W}_{1,n;I}\otimes 1_{I^c}
\subset \mathcal{W}_{1,n;\mathbb{R}}\nonumber\\
\mathcal W_I(u,P) &\equiv& \bigotimes_{I_0\in \pi} \mathcal W_{I_0}(u,P)
\qquad;\qquad \forall (u,P)\in\mathbb{R}\times\mathbb{R}_n[X]\nonumber\\
W_I(u,P) &\equiv& \bigotimes_{I_0\in \pi}W_{I_0}(u,P)
\qquad;\qquad \forall (u,P)\in\mathbb{R}\times\mathbb{R}_n[X]\label{WI(u,P)=otimes-WI0(u,P)}
\end{eqnarray}
omitting from the notations the isomorphisms implementing these
identifications.
\begin{definition}\label{df-fact-state}
A state $\varphi$ on $\mathcal{W}_{1,n;\mathbb{R}}$
is called factorizable if for every $I\subset \mathbb{R}$ bounded Borel,
for every finite partition $\pi=(I_j)_{j\in F}$ of $I$ and for every $W_I(u,P)$ as in
(\ref{WI(u,P)=otimes-WI0(u,P)}), one has:
\begin{eqnarray}\label{df-varphi-fact}
\varphi(W_I(u,P))
=\prod_{j\in F}\varphi (W_{I_j}(u,P))
\qquad;\qquad \forall (u,P)\in\mathbb{R}\times\mathbb{R}_n[X]
\end{eqnarray}
\end{definition}
The map (\ref{bb}) can be used to lift the Fock state $\varphi_{F}$ on
$\mathcal{W}_{F,1,n}$ to a state, denoted $\varphi_{0}$, on
$\mathcal{W}^{0}_{1,n}$ through the prescription
\begin{eqnarray}\label{df-varphi0}
\varphi_{0}(W^{0}(u,P)):=\varphi_{F}(W(u,P))
\end{eqnarray}
($W^{0}(u,P)\in \mathcal{W}^{0}_{1,n} \ , \
W(u,P)\in \hbox{Un}(\Gamma(\mathbb{C}))$). Then, using the maps
$\tilde z_{I}$ defined by (\ref{fact-NL-Weyl}), for each bounded
Borel set $I\subset \mathbb{R}$, one can define the state $\varphi_{I}$
on $\tilde z_{I}(\mathcal{W}^0_{1,n;I})\subset \mathcal{W}_{1,n;I}$ through
the prescription that, for each $W^0_I(u,P)\in\mathcal{W}^0_{1,n;I}$, one has
\begin{eqnarray}\label{df-varphiI}
\varphi_{I}(W_I(u,P))=\varphi_{I}(\tilde z_{I}(W^0_I(u,P))):=
\varphi_{F}(W(\hat k_I(u,P)))
\end{eqnarray}
\begin{theorem}\label{th-exist-fact-st}{\rm
Under the assumption (\ref{aI=1}), if $n=1$ then there exists 
a factorizable state $\varphi$ on
$\mathcal{W}_{1,n;\mathbb{R}}$ such that, for each bounded
Borel set $I\subset \mathbb{R}$, one has
\begin{eqnarray}\label{df-varphi}
\varphi(W_I(u,P))=\varphi_{F}(W(\hat k_I(u,P)))
\qquad;\qquad \forall (u,P)\in\mathbb{R}\times\mathbb{R}_n[X]
\end{eqnarray}
If $n\geq 2$, no such state exists.
}\end{theorem}
\begin{proof}
Let $I$ be a fixed bounded Borel set in $\mathbb{R}$ and let $\pi$
be a finite partition of $I$. 
>From Definition \ref{df-fact-state} we know that $\varphi$ is
factorizable if and only if for every $I\subset \mathbb{R}$
bounded Borel set, for every finite partition $\pi$ of $I$ and for every
$W_I(u,P)$ as in (\ref{WI(u,P)=otimes-WI0(u,P)}), (\ref{df-varphi-fact})
holds. If condition (\ref{df-varphi}) is satisfied, the identity
(\ref{df-varphi-fact}) becomes equivalent to:
\begin{eqnarray}\label{Fk-state-id}
\varphi_{F}(W(\hat k_I(u,P)))
=\prod_{I_j\in \pi}\varphi_{F} (W(\hat k_{I_j}(u,P)))
\ ; \ \forall (u,P)\in\mathbb{R}\times\mathbb{R}_n[X]
\end{eqnarray}    
Thus the statement of the theorem is equivalent to say that, for
$n=1$ the identity (\ref{Fk-state-id}) is satisfied and, for $n\geq 2$, not.
\begin{enumerate}
\item[-] \textbf{Case $n=1$}. 
For $P=a_0+a_1X$ and $u\in\mathbb{R}$, recalling the definition 
(\ref{df-hat-kI}) of $\hat k_I$, one knows that 
$$
W(\hat k_I(u,P))
=e^{i(u|I|^{\frac{1}{2}}p+a_0|I|1+a_1|I|^{\frac{1}{2}}q)}
$$
whose Fock expectation is known to be
\begin{eqnarray}\label{1}
\varphi_{F}(W(\hat k_I(u,P)))
=e^{-|I|(u^2+a_1^2)/4}  e^{ia_0|I|}=(\varphi_{F}(W(u,P)))^{|I|}
\end{eqnarray}
It follows that
\begin{eqnarray*}
\prod_{I_j\in \pi}\varphi_{F} (W(\hat k_{I_j}(u,P)))
&=&\prod_{I_j\in \pi}(\varphi_{F}(W(u,P)))^{|I_j|}
\end{eqnarray*}
Therefore, if $a_I=1$, then 
$$
\varphi_{F}(W(\hat k_I(u,P)))
=(\varphi_{F}(W(u,P)))^{|I|}
=\varphi_{F}(W(u,P))
$$
\item[-] \textbf{Case $n\geq 2$}. 
Since, for $n\geq 2$, the $1$--mode $n$--th degree Heisenberg
$*$--Lie algebra $heis_\mathbb{R}(1,n)$ contains a copy of $heis(1,2)$
(see Definition \ref{df-heis(1,n)}), the algebra
$\mathcal{W}_{1,n;\mathbb{R}}$ contains a copy of $\mathcal{W}_{1,2;\mathbb{R}}$.
Therefore the non existence of a factorizable state on
$\mathcal{W}_{1,2;\mathbb{R}}$, satisfying (\ref{df-varphi}), will imply
the same conclusion for $\mathcal{W}_{1,n;\mathbb{R}}$.
In the case $n=2$, let $P=a_0+a_1X+a_2X^2$ and $u\in\mathbb{R}$.
Then, using again $a_I=1$, (\ref{df-Fck-NLWO}) and 
(\ref{I-scal-gen-Heis1n}) one has
$$
W(\hat k_{I}(u,P))
=e^{i(|I|^{\frac{1}{2}}up+a_0|I|1+a_1|I|^{\frac{1}{2}}q+a_2q^2)}
$$
and from \cite{[AcDha]} (Theorem 2), one knows that
\begin{eqnarray*}
\varphi_{F}(W(\hat k_I(u,P)))
&=&(1-2iA)^{-\frac{1}{2}}e^{ia_0|I|}e^{\frac{4C^2(A^2+2iA)-3|M|^2}{6(1-2iA)}|I|}\\
&=&(1-2iA)^{-\frac{1}{2}}\Big(e^{ia_0}e^{\frac{4C^2(A^2+2iA)-3|M|^2}{6(1-2iA)}}\Big)^{|I|}
\end{eqnarray*}
where $A=\frac{a_2}{\sqrt{2}}$, $B=\frac{a_1}{\sqrt{2}}$, $C=\frac{u}{\sqrt{2}}$ and $M=B+iC$.
On the other hand, if $\pi \in Part_{fin}(I)$ with $|\pi|>1$, then
\begin{eqnarray*}
&&\prod_{I_j\in \pi} \varphi_{F}(W(\hat k_{I_j}(u,P))\\
&&=\prod_{I_j\in \pi}\Big((1-2iA)^{-\frac{1}{2}}\Big(e^{ia_0}e^{\frac{4C^2(A^2+2iA)-3|M|^2}{6(1-2iA)}}\Big)^{|I_j|}\Big)\\
&&=(1-2iA)^{-\frac{|\pi|}{2}}\Big(e^{ia_0}e^{\frac{4C^2(A^2+2iA)-3|M|^2}{6(1-2iA)}}\Big)^{|I|}\\
&&\neq (1-2iA)^{-\frac{1}{2}}\Big(e^{ia_0}
e^{\frac{4C^2(A^2+2iA)-3|M|^2}{6(1-2iA)}}\Big)^{|I|}
=\varphi_{F}(W(\hat k_I(u,P)))
\end{eqnarray*}
\end{enumerate}
\end{proof}

\begin{lemma}\label{iso-and-proj}{\rm
In the case $n=1$, 
the choice of the isomorphism $\hat{s}_I$ 
(see Lemma \ref{lem1}) given by
\begin{eqnarray*}
\hat{s}_I(L_{0}((\chi_I)))&=&a_{I}i|I|1\\
\hat{s}_I(L_{2}(\chi_I))&=&a_{I}i|I|^{\frac{1}{2}}p
\nonumber\\
\hat{s}_I(L_{1}(\chi_I))&=&a_{I}i|I|^{\frac{1}{2}}q
\end{eqnarray*}
gives rise to a factorizable state satisfying 
(\ref{df-varphi})
if and only if the map \\
$I\subset \mathbb R\mapsto a_I$ has the form
$$
a_I:= \frac{1}{|I|}\int_I p(s) ds
$$
for all Borel subsets $I \subseteq \mathbb R$
where $p( \ \cdot \ )$ is a locally integrable 
almost everywhere strictly positive function
on $\mathbb R$. 
In this case the factorizable state will be 
translation invariant if and only if  $p( \ \cdot \ )$ 
is a strictly positive constant.
}\end{lemma}

\begin{proof} 
In the case $n=1$, if $a_I\neq 1$, then the expression for $W(\hat k_I(u,P))$
becomes
$$ 
W(\hat k_I(u,P))
=e^{i(u|I|^{\frac{1}{2}}a^{1/2}_Ip
+a_0|I|a_I1 + a_1|I|^{\frac{1}{2}}a^{1/2}_Iq)}
$$
consequently its Fock expectation is 
\begin{eqnarray}\label{1}
\varphi_{F}(W(\hat k_I(u,P)))
=e^{-|I|a_I(u^2+a_1^2)/4}  e^{ia_0a_I|I|}
=(\varphi_{F}(W(u,P)))^{a_I|I|}
\end{eqnarray}
Therefore the factorizability condition 
(\ref{Fk-state-id}) can hold
if and only if the map $I\subset \mathbb R\mapsto a_I|I|$
is a finitely additive measure. 
In this case, by construction it will be 
absolutely continuous with respect to the 
Lebesgue measure hence there will exist a 
locally integrable almost everywhere positive 
function $p( \ \cdot \ )$ satisfying
$$
a_I|I|:= \int_I p(s) ds
\qquad ;\qquad \forall \ \hbox{Borel} \ 
I \subseteq \mathbb R
$$
$p( \ \cdot \ )$ must be almost everywhere strictly 
positive because, by Lemma \ref{lem1}, $a_I>0$ for 
any Borel set $I \subseteq \mathbb R$.
This proves the first statement of the lemma.
The second one follows because the Lebesgue measure
is the unique translation invariant positive measure
on $\mathbb R$.
\end{proof}

\end{document}